\documentclass[a4paper, 11pt]{amsart}
\usepackage[
  margin=1.9cm,
  includefoot,
  footskip=30pt,
]{geometry}
\usepackage{amsmath,amssymb,amscd}
\usepackage{amsfonts}
\usepackage[english]{babel}
\usepackage[leqno]{amsmath}
\usepackage{amssymb,amsthm}
\usepackage{mathrsfs} 
\usepackage{amscd}
\usepackage{enumerate}
\usepackage{epsfig}
\usepackage{relsize}
\usepackage{layout}
\usepackage[usenames,dvipsnames]{xcolor}
\usepackage[backref=page]{hyperref}
\usepackage{tikz}
\hypersetup{
 colorlinks,
 citecolor=Green,
 linkcolor=Red,
 urlcolor=Blue}
\usepackage[matrix,arrow,tips,curve]{xy}
\input{xy}
\xyoption{all}
\definecolor{darkgreen}{rgb}{0.0, 0.7, 0.0}
\definecolor{purple}{rgb}{0.5, 0.0, 0.5}
\definecolor{red}{rgb}{0.8, 0.2, 0.0}

\newtheorem{thm}{Theorem}[section]
\newtheorem{bthm}{Theorem}
\newtheorem{bcor}{Corollary}
\newtheorem{lemma}[thm]{Lemma}

\newtheorem{claim}[thm]{Claim}

\numberwithin{equation}{section}
\theoremstyle{definition}
\newtheorem{defi}[thm]{Definition}

\theoremstyle{remark}
\newtheorem{remark}[thm]{Remark}

\newcommand{\Pic}{\operatorname{Pic}}

\DeclareMathOperator{\Hom}{{Hom}}
\DeclareMathOperator{\Ext}{{Ext}}

\def \Im{{\rm Im}}

\def \PP{\mathbb{P}}
\def \ZZ{\mathbb{Z}}

\def \F{\mathcal F}

\def \L{\mathcal L}
\def \E{\mathcal E}
\def \G{\mathcal G}
\def\O{\mathcal O}

\def\M0{\mathcal M^0}

\DeclareMathOperator{\Ker}{{Ker}}

\begin{document}

\title[On the existence of Ulrich vector bundles on some irregular surfaces]{On the existence of Ulrich vector bundles on some irregular surfaces}

\author[A.F. Lopez]{Angelo Felice Lopez*}

\thanks{* Research partially supported by  PRIN ``Geometria delle variet\`a algebriche'' and GNSAGA-INdAM}

\address{\hskip -.43cm Dipartimento di Matematica e Fisica, Universit\`a di Roma
Tre, Largo San Leonardo Murialdo 1, 00146, Roma, Italy. e-mail {\tt lopez@mat.uniroma3.it}}


\begin{abstract} 
We establish the existence of rank two Ulrich vector bundles on surfaces that are either of maximal Albanese dimension or with irregularity $1$, under many embeddings. In particular we get the first known examples of Ulrich vector bundles on irregular surfaces of general type. Another consequence is that every surface such that either $q \le 1$ or $q \ge 2$ and its minimal model has rank one, carries a simple rank two Ulrich vector bundle. 
\end{abstract}

\maketitle

\section{Introduction}
\label{intro}

Let $X \subseteq \PP^N$ be a smooth projective variety. There are two well-known and intertwined classical ways to study the geometry of $X$, one based on the behavior of its subvarieties and the other one on the behavior of vector bundles on $X$. The latter point of view often asks for vector bundles as simple as possible, for example with few non-vanishing cohomology groups, such as in Horrocks' theorem. It is in this line of thought that Ulrich vector bundles enter the picture: a vector bundle $\E$ on $X$ is said to be Ulrich if $H^i(\E(-p))=0$ for all $i \ge 0$ and $1 \le p \le \dim X$. Ulrich vector bundles have been extensively studied in recent years, due also to their relationship with other notions such as determinantal representation, Chow forms, Clifford algebras, and so on (see for example \cite{es, b1,  ckm}). On the other hand, the basic existence question of Ulrich vector bundles on $X$ is still open even in dimension two. In the case of surfaces we now know that several classes of surfaces do carry an Ulrich vector bundle, for example K3 surfaces \cite{afo, f}, abelian surfaces \cite{b2}, bielliptic surfaces \cite{b1} and surfaces with $p_g=0, q=1$ \cite{c4, c4e}, Enriques surfaces \cite{b3, c1, c1e, bn}, del Pezzo surfaces \cite{es, ch, b1, c3, pt}, several regular surfaces \cite{c1, c1e, c2}, hypersurfaces and complete intersections \cite{hub} and several ruled surfaces \cite{acm, b1}. 

We observe that, aside from the mentioned cases, not so many results are known for the existence of Ulrich vector bundles on irregular surfaces and, as far as we know, no {\it explicit} result is known for irregular surfaces of general type. We will give below, in Theorem \ref{main} and Corollary \ref{main2}, many such examples.

In several cases the main idea is to construct a vector bundle $\E$ on the surface via a zero-dimensional Cayley-Bacharach subscheme. One strong condition is given by Riemann-Roch since one has that $\chi(\E(-p))=0$ for $p=1, 2$: the Chern classes have to satisfy two conditions \cite[Prop.~2.1]{c1}. Now this approach works directly when some suitable twist of the vector bundle has no sections. On the other hand, a beautiful idea of Faenzi \cite{f} allows to deal with many other cases: first perform an elementary transformation at a point to lower the dimension of the space of sections and then deform to get a vector bundle. This method and the necessary conditions to succeed were clearly outlined in \cite{f} and \cite{c2}. 

In the study of a smooth irregular surface $S$, an important tool is given by its Albanese map. Recall that the Albanese variety of $S$ is the abelian variety ${\rm Alb}(S) = H^0(S,\Omega^1_S)^* /  H^1(S, \mathbb Z)$ and that integration of one-forms determines a morphism, the Albanese mapping,  $a: S \to {\rm Alb}(S)$. In particular $S$ is said to be of {\it maximal Albanese dimension} if $\dim a(S)=2$.

\eject

Using the same kind of ideas in \cite{f} and \cite{c2}, in this paper we consider irregular surfaces and prove the following:

\begin{bthm}
\label{main}
Let $S$ be a smooth irreducible surface such that either $S$ is of maximal Albanese dimension or $q(S)=1$. Let $H$ be a very ample divisor on $S$ satisfying:
\begin{itemize}
\item [(a)] $H^1(\O_S(H))=H^2(\O_S(H))=0$

\hskip -1.7cm and, when $\chi(\O_S) \ge 1$, also

\item [(b)] $H^2 \ge H \cdot K_S + 2 q(S)-2$, and
\item [(c)] $H^0(\O_S(2K_S-H))=0$.
\end{itemize}
Then there exists a rank two Ulrich vector bundle $\E$ for the pair $(S,H)$. 

Furthermore $\E$ is simple if either $\chi(\O_S) \ge 1$ or $\chi(\O_S) = 0$ and $q(S) \ge 2$ or $\chi(\O_S) = 0, q(S) =1$, $h^0(\O_S(-K_S)) = 0$ and $h^0(\O_S(H-K_S)) \le h^0(\O_S(H))$.
\end{bthm}

As far as we know, aside from the case of curves, it is still an open question whether every smooth projective variety carries, in a suitable embedding, an Ulrich vector bundle. In the case of complex surfaces, Coskun and Huizenga \cite[Thm.~1.2]{cohu} show the existence of Ulrich bundles of rank two in a sufficiently ample embedding. 

We prove existence for surfaces, in a suitable embedding, in many cases including the rank one case.

\begin{bcor}
\label{main2}
Let $S$ be a smooth irreducible surface such that either $q(S) \le 1$ or $q(S) \ge 2$ and its minimal model $S_0$ has $\Pic(S_0) \cong \ZZ$. Then there exists a simple rank two Ulrich vector bundle $\E$ for the pair $(S,H)$, where $H$ is a suitable very ample divisor. 
\end{bcor}

Let us briefly comment on the hypotheses of Theorem \ref{main}. For proofs see Remark \ref{ipo}. Property (a) holds for example if $H-K_S$ is big and nef, property (b) if $H^1(\O_S(H))=0$ and $p_g(S) \le 5$. Property (c) is satisfied for example if $S$ is not of general type. Moreover if $H$ is sufficiently ample then (a), (b) and (c) are satisfied on any surface. 

Since there are many surfaces of maximal Albanese dimension or with $q(S)=1$ that are of general type (see for example \cite{pml}), we get, at least when $H$ is sufficiently ample (so that (a), (b) and (c) are satisfied), that they carry a rank two Ulrich vector bundle.

In the last section we show that there exists a rank two Ulrich vector bundle on some specific examples including Weierstrass fibrations.

Throughout the whole paper we work over an algebraically closed field $k$ of characteristic zero. A surface is by definition a $2$-dimensional projective scheme over $k$. We say that a property $\mathcal P$ is satisfied by a general element of a variety $X$ if $\mathcal P$ holds on every point of a non-empty open subset of $X$.

\section{Elementary transformation at a point}

\begin{defi}
\label{elem} 
Let $S$ be a smooth irreducible surface, let $\F$ be a locally free sheaf of rank two on $S$ with $h^0(\F) \ge 1$, let $P \in S$ be a point and let $\varphi \in \Hom(\F, \O_{P})$ such that $\varphi \neq 0$. We say that  $\F_{\varphi} := \Ker (\varphi)$ is the {\it elementary transformation of $\F$ at $P$}. If $P$ is general in $S$ and $\varphi$ is general in $\Hom(\F, \O_{P})$ we say that $\F_{\varphi}$ is a {\it general elementary transformation of $\F$}.
\end{defi}
By the above definition we have an exact sequence
\begin{equation}
\label{ffi}
\xymatrix{0 \ar[r] & \F_{\varphi} \ar[r]^{j} & \F \ar[r]^{\varphi} & \O_{P} \ar[r] & 0.}
\end{equation}
Moreover it follows exactly as in \cite[Construction 4.1 and Lemma 4.2]{c2} (note that for \eqref{ffi2} the hypothesis $h^0(\F(-H))=0$ is not necessary) that if $\F_{\varphi}$ is a general elementary transformation of $\F$ then
\begin{equation}
\label{ffi2}
c_1(\F_{\varphi}) = c_1(\F), \ \ c_2(\F_{\varphi}) = c_2(\F)+1, \ \ h^0(\F_{\varphi}) = h^0(\F) -1 \ \ \mbox{and} \ \ H^1(\F_{\varphi}) = H^1(\F).
\end{equation}
As in \cite[Lemma 2]{f} we have

\begin{lemma}
\label{simp} 
Let $S$ be a smooth irreducible surface, let $\F$ be a locally free sheaf of rank two on $S$ with $h^0(\F) \ge 1$ and let $\F_{\varphi}$ be an elementary transformation of $\F$. If $\F$ is simple then so is $\F_{\varphi}$.
\end{lemma}
\begin{proof} 
First we show that $\Hom( \F, \F_{\varphi}) = 0$. Indeed if there is $0 \neq \psi \in \Hom( \F, \F_{\varphi})$, then $0 \neq j \circ \psi \in \Hom(\F,\F)$. As $\F$ is simple we get that $j \circ \psi $ is surjective and therefore so is $j$, a contradiction. Applying 
$\Hom(-,\F_{\varphi})$ to \eqref{ffi} we get an inclusion of $\Hom(\F_{\varphi}, \F_{\varphi})$ into $\Ext^1(\O_{P},\F_{\varphi})$, so it remains to prove that $\dim \Ext^1(\O_{P},\F_{\varphi})=1$. Since $\F$ is locally free we have by Serre duality that $\Ext^i(\O_{P},\F) \cong H^{2-i}(\F^*(K_S) \otimes \O_{P})^*=0$ for $i=0,1$. Applying $\Hom(\O_{P}, -)$ to \eqref{ffi} we therefore get that $\Hom(\O_{P},\O_{P}) \cong \Ext^1(\O_{P},\F_{\varphi})$ and we are done.
\end{proof}

\begin{lemma}
\label{riduz} 
Let $S$ be a smooth irreducible surface and let $H$ be a very ample divisor on $S$. Let $\F$ be a locally free sheaf of rank two on $S$ satisfying the following properties:
\begin{itemize}
\item [(i)] $c_1(\F) \equiv H+K_S$;
\item [(ii)] $c_2(\F) = \frac{1}{2} H \cdot (H+K_S) + 2 \chi(\O_S)-h^0(\F)$;
\item [(iii)] $H^1(\F) = H^2(\F) = 0$;
\item [(iv)] $\dim \Ext^2(\F,\F) = p_g(S)$;
\item [(v)] $\F$ is simple.
\end{itemize}
If $h^0(\F) \ge 1$ and $\F_{\varphi}$ is a general elementary transformation of $\F$, then $\F_{\varphi}$ also satisfies properties (i)-(v).
\end{lemma}
\begin{proof} 
By \eqref{ffi} and \eqref{ffi2} we get immediately that $\F_{\varphi}$ satisfies properties (i)-(iii). Now $\F_{\varphi}$ satisfies (iv) by \cite[Prop.~4.3(2)]{c2} and (v) by Lemma \ref{simp} .
\end{proof}

\section{The starting vector bundle}

We first construct, as in \cite{b2}, a vector bundle that will be the starting point of an inductive procedure.

\begin{lemma}
\label{e0} 
Let $S$ be a smooth irreducible surface such that either $S$ is of maximal Albanese dimension or $q(S)=1$. Let $H$ be a very ample divisor on $S$ satisfying:
\begin{itemize}
\item [(a)] $H^1(\O_S(H))=H^2(\O_S(H))=0$

\hskip -1.7cm and, when $\chi(\O_S) \ge 1$, also

\item [(b)] $H^2 \ge H \cdot K_S + 2 q(S)-2$, and
\item [(c)] $H^0(\O_S(2K_S-H))=0$.
\end{itemize}

Let $\eta \in Pic^0(S)$ be general. Then there exists a rank two vector bundle $\G_{\eta}$ on $S$ with 
\begin{equation}
\label{giusto}
c_1(\G_{\eta}) = H + K_S + 2 \eta, \ \ h^0(\G_{\eta}) = h^0(\G_{\eta}(-2\eta)) = \chi(\O_S)
\end{equation}
and such that both $\G_{\eta}$ and $\G_{\eta}(-2\eta)$ satisfy (i)-(iii) of Lemma \ref{riduz}, (iv) if $\chi(\O_S) \ge 1$ and also (v) if either $p_g(S) \ge 1$ or $h^0(\O_S(-K_S))=0$ and $h^0(\O_S(H-K_S)) \le h^0(\O_S(H))$.
\end{lemma}
\begin{proof} 
Set $N+1 = h^0(\O_S(H))$. By (a) and Riemann-Roch we find that 
\begin{equation}
\label{ineq1}
\frac{1}{2} H \cdot (H-K_S) + \chi(\O_S) = N+1
\end{equation}
whence
\begin{equation}
\label{ineq}
\hbox{(b) is equivalent to} \ \ N+1 \ge p_g(S).
\end{equation}
Let us also record that 
\begin{equation}
\label{effe1}
H^0(\O_S(K_S-H))=H^1(\O_S(K_S-H))=0.
\end{equation}
Indeed $H^i(\O_S(K_S-H)) = H^{2-i}(\O_S(H))^*=0$ for $i=0,1$ by (a).

We will now divide the proof into several claims. 
\begin{claim}
\label{uno}
The following hold:
\begin{enumerate}[(i)]
\item \label{uno-i} $H^i(\O_S(\pm \eta)) = 0 \ \mbox{for} \ i = 0, 1$;
\item \label{uno-ii} $h^2(\O_S(\pm \eta)) = \chi(\O_S)$;
\item \label{uno-iii} $h^0(\O_S(H \pm \eta)) = N+1$;
\item \label{uno-iv} $H^i(\O_S(H \pm \eta)) = 0 \ \mbox{for} \ i = 1, 2$.
\end{enumerate}
\end{claim}
\begin{proof}
To see \eqref{uno-i} observe that $H^0(\O_S(\pm \eta)) = 0$ since $\eta \neq \O_S$. As for
\begin{equation}
\label{first}
H^1(\O_S(\pm \eta)) = 0
\end{equation}
we distinguish the two cases of the hypothesis. 

If $S$ is of maximal Albanese dimension then \eqref{first} follows by generic vanishing, that holds over the complex numbers by \cite[Thm.~1]{gl} and over $k$ by \cite[Cor.~3.2]{hc} and Grauert-Riemenschneider vanishing \cite[Cor.~1-2-4]{kmm}. 

Now assume that $q(S)=1$. Then the Albanese variety of $S$ is an elliptic curve $E$ and the Albanese map $a: S \to E$ is surjective and has connected fibers by \cite[Prop.~V.15]{b4}. We claim that $H^0(R^1a_*\O_S)=0$. Indeed the Leray spectral sequence gives the exact sequence
$$0 \to H^1(a_*\O_S) \to H^1(\O_S) \to H^0(R^1a_*\O_S) \to 0$$
and since $h^1(a_*\O_S)=h^1(\O_E)=1=h^1(\O_S)$, we get that $H^0(R^1a_*\O_S)=0$. As in \cite[Proof of Cor.~3.3]{c4} it follows by the existence of a Poincar\'e line bundle and semicontinuity that the subsets
$$Y^{\pm} := \{L \in \Pic^0(E) : h^0((R^1a_*\O_S)(\pm L)) > 0\}$$
are closed and $\O_E \not\in Y^{\pm}$, whence $h^0((R^1a_*\O_S)(\pm \eta_0)) = 0$ for $\eta_0 \in \Pic^0(E)$ general. Let $\eta' = a^* \eta_0$ and set $\L = \pm \eta'$. The Leray spectral sequence gives the exact sequence
\begin{equation}
\label{ler}
0 \to H^1(a_*\L) \to H^1(\L) \to H^0(R^1a_*\L) \to 0.
\end{equation}
Since $a_*\L = \O_E(\pm \eta_0)$ and $R^1a_*\L = (R^1a_*\O_S) (\pm \eta_0)$ we get 
$H^1(a_*\L)=H^1(\O_E(\pm \eta_0))=0$ and $H^0(R^1a_*\L) = H^0((R^1a_*\O_S)(\pm \eta_0)) = 0$ so that \eqref{ler} gives that $H^1(\L)=0$ and \eqref{first} follows. 

Hence \eqref{uno-i} is proved. Now \eqref{uno-ii} follows by \eqref{uno-i} and Riemann-Roch.

To see \eqref{uno-iv} notice that, as above, the subsets
$$Y_i^{\pm} := \{L \in \Pic^0(S) : h^i(\O_S(H \pm L)) > 0\}$$
are closed. Since, by hypothesis (a), $\O_S \not\in Y_i^{\pm}$ for $i = 1, 2$, we get that $H^i(\O_S(H \pm \eta)) = 0$ for $i = 1, 2$ and $\eta \in \Pic^0(S)$ general. That is \eqref{uno-iv} holds. Now \eqref{uno-iii} follows by \eqref{uno-iv}, (a) and Riemann-Roch since $h^0(\O_S(H \pm \eta)) = \chi(\O_S(H \pm \eta)) = \frac{1}{2} H \cdot (H-K_S) + \chi(\O_S) = h^0(\O_S(H)) = N+1$. This proves Claim \ref{uno}.
\end{proof}

We now construct a suitable $0$-dimensional subscheme $Z$ of $S$ having the necessary properties.

\begin{claim}
\label{due}
Let $C \in |H|$ be a smooth irreducible curve defined by a section $s \in H^0(\O_S(H))$ and let $P_i, 1 \le i \le N+1$ be general points on $C$. Then the following hold:
\begin{enumerate}[(i)]
\item \label{due-i} $h^0(\O_C(K_S)) = p_g(S)$
\item \label{due-ii} $Z = P_1 + \ldots + P_{N+1} \ \mbox{satisfies the Cayley-Bacharach property with respect to} \ |H|$;
\item \label{due-iii} $H^0({\mathcal I}_{Z/S}(H))= k s$;
\item \label{due-iv} $H^i({\mathcal I}_{Z/S}(H \pm \eta)) = 0 \ \mbox{for} \ i \ge 0$;
\item \label{due-v} If $\chi(\O_S) \ge 1$ then $H^0({\mathcal I}_{Z/S}(K_S))=0$;
\item \label{due-vi} If $\chi(\O_S) \ge 1$ then $h^0(\O_C(H)(-2Z))=0$;
\item \label{due-vii} If either $K_S \neq 0$ and $p_g(S) \ge 1$ or $h^0(\O_S(-K_S))=0$ and $h^0(\O_S(H-K_S)) \le h^0(\O_S(H))$, then $H^0({\mathcal I}_{Z/S}(H-K_S))=0$.
\end{enumerate}
\end{claim}
\begin{proof}
By \eqref{effe1} and the exact sequence
$$0 \to \O_S(K_S-H) \to \O_S(K_S) \to \O_C(K_S) \to 0$$
we get that
$$h^0(\O_C(K_S)) = h^0(\O_S(K_S)) = p_g(S)$$
thus giving \eqref{due-i}. To see \eqref{due-ii}  let $V = \Im \{H^0(\O_S(H)) \to H^0(\O_C(H))\}$, so that $C \subset \PP V = \PP^{N-1}$. Now let $i \in \{1, \ldots, N+1\}$ and let $\sigma \in H^0(\O_S(H))$ be such that $\sigma (P_j)=0$ for all $j \neq i$. If $\sigma_{|C} \neq 0$ then it defines a hyperplane in $\PP V$ passing through $N$ general points of $C$, a contradiction. Therefore $\sigma_{|C} = 0$ and then $\sigma (P_i)=0$. This proves \eqref{due-ii} and also \eqref{due-iii}.

Next we prove \eqref{due-iv}. Using Claim \ref{uno}\eqref{uno-i} and Claim \ref{uno}\eqref{uno-iii} in the exact sequences
$$0 \to \O_S(\pm \eta) \to \O_S(H \pm \eta) \to \O_C(H \pm \eta) \to 0$$
we get that $h^0(\O_C(H \pm \eta)) = h^0(\O_S(H \pm \eta)) = N+1$. Moreover, as $Z$ consists of $N+1$ general points of $C$, we have that 
%\begin{equation}
%\label{altro}
$$h^0({\mathcal I}_{Z/C}(H \pm \eta)) = h^0(\O_C(H \pm \eta))-N-1 = 0.$$
%\end{equation}
Applying the latter vanishing in the exact sequences 
%\begin{equation}
%\label{id}
$$0 \to \O_S(\pm \eta) \to {\mathcal I}_{Z/S}(H \pm \eta) \to {\mathcal I}_{Z/C}(H \pm \eta) \to 0$$
%\end{equation}
and using Claim \ref{uno}\eqref{uno-i}, we get that $H^0({\mathcal I}_{Z/S}(H \pm \eta)) = 0$. Now Claim \ref{uno}\eqref{uno-iii} and Claim \ref{uno}\eqref{uno-iv} and the exact sequences 
$$0 \to {\mathcal I}_{Z/S}(H \pm \eta) \to \O_S(H \pm \eta) \to \O_Z(H \pm \eta) \to 0$$
show that $H^i({\mathcal I}_{Z/S}(H \pm \eta)) = 0$ for $i \ge 1$. Thus \eqref{due-iv} is proved.

As for the proof of \eqref{due-v} assume that $\chi(\O_S) \ge 1$. Since $h^0(\O_C(K_S)) = p_g(S)$ by \eqref{due-i},
we see that, as $Z$ is general, $h^0({\mathcal I}_{Z/C}(K_S)) = \max\{0, h^0(\O_C(K_S)) - N -1\} = 0$ by \eqref{ineq}. Then, from the exact sequence
$$0 \to \O_S(K_S-H) \to {\mathcal I}_{Z/S}(K_S) \to {\mathcal I}_{Z/C}(K_S) \to 0$$
and \eqref{effe1} we deduce that $H^0({\mathcal I}_{Z/S}(K_S))=0$, that is \eqref{due-v}.

To see \eqref{due-vi} suppose that $\chi(\O_S) \ge 1$. Notice that the inequality 
\begin{equation}
\label{ine}
2(N+1) \ge N+q(S)
\end{equation}
holds. Indeed \eqref{ine} is equivalent to $N+1 \ge q(S)-1$ and, by \eqref{ineq1}, it is also equivalent to 
$$H^2 + 2 p_g(S) \ge H \cdot K_S + 4 q(S)-4.$$ 
But $\chi(\O_S) \ge 1$, hence $p_g(S) \ge q(S)$ and then (b) gives
$$H^2 + 2 p_g(S) \ge H \cdot K_S + 4 q(S) - 2$$
so that \eqref{ine} is proved. By (a) and the exact sequence
$$0 \to \O_S \to  \O_S(H) \to  \O_C(H) \to 0$$
we get that $h^0(\O_C(H)) = N + q(S)$ and therefore, being $Z$ is general, we find by \eqref{ine} that 
%\begin{equation}
%\label{fino}
$$h^0(\O_C(H)(-2Z))=\max\{0, N + q(S)-(2N+2)\}=0$$
%\end{equation}
that is \eqref{due-vi} holds.

In order to see \eqref{due-vii} assume first that  $h^0(\O_S(-K_S))=0$ and $h^0(\O_S(H-K_S)) \le h^0(\O_S(H))$. Consider the exact sequence
$$0 \to \O_S(-K_S) \to \O_S(H-K_S) \to \O_C(H-K_S) \to 0.$$
Let $W = \Im \{H^0(\O_S(H-K_S)) \to H^0(\O_C(H-K_S))\}$, so that $\dim W = h^0(\O_S(H-K_S)) \le N+1$. If $W=\{0\}$ then $H^0(\O_S(H-K_S))=0$ and therefore also $H^0({\mathcal I}_{Z/S}(H-K_S))=0$, that is \eqref{due-vii} holds. If $\dim W \ge 1$ let $\varphi: C \to \PP W$ be the morphism defined by $W$. Let $\sigma \in H^0(\O_S(H-K_S))$ be such that $\sigma (P_i)=0$ for all $1 \le i \le N+1$. If $\sigma_{|C} \neq 0$ then it defines a hyperplane in $\PP W$ passing through $N+1$ general points of $\varphi(C)$, a contradiction since by definition $\varphi(C)$ is not contained in a hyperplane. Therefore $\sigma_{|C} = 0$, whence $\sigma = 0$. This gives \eqref{due-vii} in this case.

Now assume that $K_S \neq 0$ and $p_g(S) \ge 1$. Pick $\tau \in H^0(\O_S(K_S))$ such that $\tau \neq 0$. Let $\sigma \in H^0({\mathcal I}_{Z/S}(H - K_S))$. If $\sigma \neq 0$ then $0 \neq \sigma \tau \in H^0({\mathcal I}_{Z/S}(H))$, hence, by \eqref{due-iii}, $\sigma \tau = \lambda s$, for some $\lambda \in k^*$. Let $D$ be the divisor associated to $\tau$. Note that $D$ is effective non-zero because $K_S \neq 0$. But $D \subseteq C$ and therefore $D=C$, that is $K_S \sim H$, giving by (a) the  contradiction $0 = h^1(\O_S(H)) = h^1(\O_S(K_S)) = q(S)$. Therefore $\sigma = 0$ and \eqref{due-vii} is proved.
Thus the proof of Claim \ref{due} is complete.
\end{proof}

Keeping notation as in Claim \ref{due}, we now see that the $0$-dimensional subscheme $Z$ of $S$ gives rise to a vector bundle whose properties are listed below.

\begin{claim}
\label{tre}
There exists a rank two vector bundle $\F$ on $S$ sitting in two exact sequences
\begin{equation}
\label{prima}
\xymatrix{0 \ar[r] & \O_S \ar[r] & \F \ar[r]& {\mathcal I}_{Z/S}(H - K_S) \ar[r] & 0}
\end{equation}
and
\begin{equation}
\label{altra}
\xymatrix{0 \ar[r] & \O_C(Z) \ar[r] & \F \otimes \O_C \ar[r]& \O_C(H-K_S)(-Z) \ar[r] & 0.}
\end{equation}
Moreover $\F$ satisfies:
\begin{enumerate}[(i)]
\item \label{tre-i} $c_1(\F) = H - K_S$;
\item \label{tre-ii} $c_2(\F)=N+1$;
\item \label{tre-iii} $H^0(\F(K_S-H))=0$;
\item \label{tre-iv} If $\chi(\O_S) \ge 1$ then $H^0(\F(2K_S-H))=0$;
\item \label{tre-v} If $\chi(\O_S) \ge 1$ then $\dim \Ext^2(\F,\F) = p_g(S)$;
\item \label{tre-vi} If either $K_S \neq 0$ and $p_g(S) \ge 1$ or $h^0(\O_S(-K_S))=0$ and $h^0(\O_S(H-K_S)) \le h^0(\O_S(H))$, then $\F$ is simple.
\end{enumerate}
\end{claim}
\begin{proof}
By Claim \ref{due}\eqref{due-ii} we have that $Z$ satisfies the Cayley-Bacharach property with respect to $|H|$. As is well-known \cite[Thm.~5.1.1]{hl}, $Z$ gives rise to a rank two vector bundle $\F$, with $c_1(\F) = H - K_S$, $c_2(\F)=N+1$, together with a section $\gamma \in H^0(\F)$ whose zero locus is $Z$ and to an exact sequence
$$\xymatrix{0 \ar[r] & \O_S \ar[r]^{\gamma} & \F \ar[r]& {\mathcal I}_{Z/S}(H - K_S) \ar[r] & 0.}$$
Also observe that, since $\det (\F \otimes \O_C) = \O_C(H - K_S)$, we get an exact sequence
$$\xymatrix{0 \ar[r] & \O_C(Z) \ar[r]^{\gamma_{|C}} & \F \otimes \O_C \ar[r]& \O_C(H-K_S)(-Z) \ar[r] & 0.}$$
Thus we have shown \eqref{tre-i}, \eqref{tre-ii} and the existence of \eqref{prima} and \eqref{altra}.

Tensoring \eqref{prima} by $\O_S(K_S-H)$ we get the exact sequence
$$0 \to \O_S(K_S-H) \to \F(K_S-H) \to {\mathcal I}_{Z/S} \to 0$$
whence \eqref{effe1} gives that $H^0(\F(K_S-H))=0$, that is \eqref{tre-iii}. To see \eqref{tre-iv} assume that $\chi(\O_S) \ge 1$. Tensoring \eqref{prima} by $\O_S(2K_S-H)$ we get the exact sequence
$$0 \to \O_S(2K_S-H) \to \F(2K_S-H) \to {\mathcal I}_{Z/S}(K_S) \to 0.$$
Hence (c) and Claim \ref{due}\eqref{due-v} give that $H^0(\F(2K_S-H))=0$, that is \eqref{tre-iv}.

To see \eqref{tre-v} assume again that $\chi(\O_S) \ge 1$. Since the inequality $\dim \Ext^2(\F,\F) \ge p_g(S)$ holds for every rank two vector bundle, we just need to prove that 
\begin{equation}
\label{ext3}
\dim \Ext^2(\F,\F) \le p_g(S).
\end{equation}
Tensoring \eqref{altra} by $\O_C(K_S)(-Z)$ gives the exact sequence
\begin{equation}
\label{fin}
0 \to \O_C(K_S) \to \F \otimes \O_C(K_S)(-Z) \to \O_C(H)(-2Z) \to 0.
\end{equation}
Now using Claim \ref{due}\eqref{due-i} and Claim \ref{due}\eqref{due-vi}, we get from \eqref{fin} that 
\begin{equation}
\label{finu}
h^0({\mathcal I}_{Z/C} \otimes \F (K_S))=h^0(\F \otimes \O_C(K_S)(-Z))= h^0(\O_C(K_S)) = p_g(S).
\end{equation}
Then from the exact sequence
$$0 \to \F(K_S-H) \to {\mathcal I}_{Z/S}\otimes \F(K_S) \to {\mathcal I}_{Z/C} \otimes \F (K_S) \to 0$$
and \eqref{tre-iii} we deduce using \eqref{finu} that 
\begin{equation}
\label{finio}
h^0( {\mathcal I}_{Z/S} \otimes \F(K_S)) \le h^0({\mathcal I}_{Z/C} \otimes \F (K_S)) = p_g(S).
\end{equation}
Since $\det \F = \O_S(H-K_S)$ we have that
$$\F^*(K_S) \cong  \F(2K_S-H)$$ 
whence we get by \eqref{tre-iv} that 
$H^0(\F^*(K_S))=0$. From \eqref{prima} we also have the exact sequence
\begin{equation}
\label{sec}
0 \to \F^*(K_S) \to \F \otimes \F^*(K_S) \to {\mathcal I}_{Z/S} \otimes \F(K_S) \to 0
\end{equation}
and therefore we find from \eqref{finio} that
$$\dim \Ext^2(\F,\F) = h^0(\F \otimes \F^*(K_S)) \le h^0( {\mathcal I}_{Z/S} \otimes \F(K_S)) \le p_g(S)$$
thus proving \eqref{ext3}. Hence \eqref{tre-v} is proved. Finally to see  \eqref{tre-vi} assume that either $K_S \neq 0$ and $p_g(S) \ge 1$ or $h^0(\O_S(-K_S))=0$ and $h^0(\O_S(H-K_S)) \le h^0(\O_S(H))$. Since $H^0({\mathcal I}_{Z/S}(H-K_S))=0$ by Claim \ref{due}\eqref{due-vii}, we deduce from \eqref{prima} that $h^0(\F)=1$. 
Tensoring \eqref{prima} by $\F^* \cong \F(K_S-H)$ we get the exact sequence
\begin{equation}
\label{simple}
0 \to \F^* \to \F \otimes \F^* \to {\mathcal I}_{Z/S} \otimes \F \to 0.
\end{equation}
Now $H^0(\F^*) = H^0(\F(K_S-H)) = 0$ by \eqref{tre-iii} and therefore, using \eqref{simple} we have
$$h^0(\F \otimes \F^*) \le h^0( {\mathcal I}_{Z/S} \otimes \F) \le h^0(\F) = 1$$
that is $\F$ is simple. This proves \eqref{tre-vi} and concludes the proof of Claim \ref{tre}.
\end{proof}
We now proceed to end the proof of Lemma \ref{e0}.
 
Let $\F$ be the vector bundle found in Claim \ref{tre} and let $\G_{\eta} = \F(K_S+\eta)$. To ease notation, set, from now on, $\G = \G_{\eta}$. 

We first record that, by Claim \ref{tre}\eqref{tre-i}, Claim \ref{tre}\eqref{tre-ii} and \eqref{ineq1}, we have 
\begin{equation}
\label{cl}
c_1(\G) = H + K_S + 2 \eta, \ \ c_2(\G) = \frac{1}{2} H \cdot (H+K_S) + \chi(\O_S)
\end{equation}
and by \eqref{prima} there are two exact sequences
\begin{equation}
\label{est}
0 \to \O_S(K_S + \eta) \to \G \to {\mathcal I}_{Z/S}(H + \eta) \to 0
\end{equation}
\begin{equation}
\label{est1}
0 \to \O_S(K_S -\eta) \to \G(-2 \eta) \to {\mathcal I}_{Z/S}(H - \eta) \to 0.
\end{equation}
By Serre duality, Claim \ref{uno}\eqref{uno-i} and Claim \ref{uno}\eqref{uno-ii} we find that $h^0(\O_S(K_S \pm \eta)))=h^2(\O_S(\mp \eta))=\chi(\O_S)$ and $h^i(\O_S(K_S \pm \eta)))=h^{2-i}(\O_S(\mp \eta))=0$ for $i =1, 2$.
We therefore deduce by \eqref{est}, \eqref{est1} and Claim \ref{due}\eqref{due-iv} that 
\begin{equation}
\label{hje}
h^0(\G) = h^0(\G(-2 \eta)) = \chi(\O_S), \ \ H^i(\G) = H^i(\G(-2 \eta)) = 0  \ \mbox{for} \ i \ge 1.
\end{equation}

From \eqref{cl} and \eqref{hje} we see that \eqref{giusto} is proved and also that $\G$ and $\G(-2\eta)$ satisfy (i)-(iii) of Lemma \ref{riduz}. Also (iv) of Lemma \ref{riduz} is satisfied by $\G$ and $\G(-2\eta)$ when $\chi(\O_S) \ge 1$ since 
$$\dim \Ext^2(\G,\G) = \dim \Ext^2(\G(-2 \eta),\G(-2 \eta)) = \dim \Ext^2(\F,\F) = p_g(S)$$ 
by Claim \ref{tre}\eqref{tre-v}. 

Finally let us prove that (v) of Lemma \ref{riduz} holds for $\G$ and $\G(-2\eta)$ if either $p_g(S) \ge 1$ or $h^0(\O_S(-K_S))=0$ and $h^0(\O_S(H-K_S)) \le h^0(\O_S(H))$. 

If $K_S \neq 0$ we obtain that Claim \ref{tre}\eqref{tre-vi} holds, that is $\F$ is simple. Therefore so are $\G$ and $\G(-2 \eta)$, whence they satisfy (v) of Lemma \ref{riduz}.

If $K_S = 0$ then $S$ is an Abelian surface and $\G$ is the same vector bundle constructed in \cite[Thm.~1]{b2}. This is simple by \cite[Rmk.~2]{b2}, and, of course, so is $\G(-2 \eta)$. Therefore (v) of Lemma \ref{riduz} is again satisfied.

This concludes the proof of Lemma \ref{e0}.
\end{proof}

\section{Proof of Theorem \ref{main} and Corollary \ref{main2}}

\renewcommand{\proofname}{Proof of Theorem \ref{main}}
\begin{proof} 
Let $\eta \in Pic^0(S)$ be general. Set $\chi = \chi(\O_S)$ and notice that $\chi \ge 0$: this is obvious for $q(S)=1$ and follows by \cite[Cor.]{gl} when $S$ has maximal Albanese dimension. 

For each $i = 0, \ldots, \chi$ we claim that we can construct a rank two vector bundle $\G_i$ such that:
\begin{enumerate}[\hspace{.4cm}(A)]
\item $h^0(\G_i) = h^0(\G_i(-2\eta)) = \chi-i$;
\item $\G_i$ and $\G_i(-2\eta)$ satisfy (i)-(iii) of Lemma \ref{riduz};
\item If $i \ge 1$ then both $\G_i$ and $\G_i(-2\eta)$ satisfy (iv) and (v) of Lemma \ref{riduz}.
\end{enumerate}
To see the claim let $\G_{\eta}$ be the vector bundle constructed in Lemma \ref{e0} and set $\G_0 = \G_{\eta}$. Then $h^0(\G_0) = h^0(\G_0(-2\eta)) = \chi$. Also $\G_0$ and $\G_0(-2\eta)$ satisfy (i)-(iii) of Lemma \ref{riduz}, (iv) if $\chi \ge 1$ and also (v) if either $p_g(S) \ge 1$ or $h^0(\O_S(-K_S))=0$ and $h^0(\O_S(H-K_S)) \le h^0(\O_S(H))$. Now assume that $1 \le i \le \chi$ and suppose we have constructed $\G_{i-1}$ as in the claim. Since $\chi \ge 1$ we have that $p_g(S) \ge1$, whence, when $i =1$, we get that $\G_0$ and $\G_0(-2\eta)$ satisfy (iv) and (v) of Lemma \ref{riduz}. 
Therefore $\G_{i-1}$ and $\G_{i-1}(-2\eta)$ satisfy (i)-(v) of Lemma \ref{riduz}. Now $h^0(\G_{i-1}) = \chi-i+1 \ge 1$, whence, by Lemma \ref{riduz}, a general elementary transformation $(\G_{i-1})_{\varphi}$ of $\G_{i-1}$ also satisfies properties (i)-(v). As in \cite[Proof of Lemma 4]{f} or \cite[Proof of Thm.~1.1]{c2}, it follows by \cite[Thm.~1.4 and Cor.~1.5]{a} that there exists a deformation of $(\G_{i-1})_{\varphi}$, over an integral base, whose general element is a rank two vector bundle $\G_i$. Then by semicontinuity (using either \cite[Prop.~6.4]{ha} or \cite[Thm.~0.3]{mu}) $\G_i$ satisfies (iii)-(v) of Lemma \ref{riduz}. By \eqref{ffi2} we have
\begin{equation}
\label{c1}
c_1(\G_i) = c_1((\G_{i-1})_{\varphi}) = c_1(\G_{i-1})  \equiv H+K_S
\end{equation}
so that $\G_i$ satisfies (i) and
\begin{equation}
\label{c2} 
c_2(\G_i) = c_2((\G_{i-1})_{\varphi}) = c_2(\G_{i-1}) + 1 = \frac{1}{2} H \cdot (H+K_S) + \chi + i
\end{equation}
and therefore, by Riemann-Roch
\begin{equation}
\label{c3}
h^0(\G_i) = \chi (\G_i) = 2\chi + \frac{1}{2} c_1(\G_i) \cdot (c_1(\G_i) - K_S) - c_2(\G_i) = \chi-i
\end{equation}
whence $\G_i$ also satisfies (ii) of Lemma \ref{riduz}.

Now also $h^0(\G_{i-1}(-2\eta)) = \chi-i+1 \ge 1$, whence, by Lemma \ref{riduz}, a general elementary transformation $(\G_{i-1}(-2\eta))_{\varphi}$ of $\G_{i-1}(-2\eta)$ also satisfies properties (i)-(v). Since $\Hom(\G_{i-1},\O_{P}) \cong \Hom(\G_{i-1}(-2\eta),\O_{P})$ we see that  $(\G_{i-1}(-2\eta))_{\varphi} \cong  (\G_{i-1})_{\varphi}(-2\eta)$, whence $\G_i(-2\eta)$ is a generalization of $(\G_{i-1}(-2\eta))_{\varphi}$. Again by semicontinuity $\G_i(-2\eta)$ satisfies (iii)-(v) of Lemma \ref{riduz}. By \eqref{c1} we have that $c_1(\G_i(-2\eta)) \equiv c_1(\G_i) \equiv H+K_S$ so that $\G_i(-2\eta)$ satisfies (i) of Lemma \ref{riduz} and, by \eqref{c2}, $c_2(\G_i(-2\eta)) = c_2(\G_i) = \frac{1}{2} H \cdot (H+K_S) + \chi + i$.
Also $h^0(\G_i(-2\eta)) = \chi (\G_i(-2\eta)) = \chi (\G_i) = \chi-i$, whence $\G_i(-2\eta))$ also satisfies (ii) of Lemma \ref{riduz}. This proves the claim.

Finally let us prove that $\E:= \G_{\chi}(H)$ is an Ulrich vector bundle. To this end observe that we have found, by (A), that 
$$H^0(\G_{\chi}) = H^0(\G_{\chi}(-2 \eta))=0,$$
by \eqref{c2} that 
$$c_2(\G_{\chi})  = \frac{1}{2} H \cdot (H+K_S) + 2 \chi(\O_S),$$
and by \eqref{c1} that 
$$c_1(\G_{\chi}) = c_1(\G_0) = c_1(\G_{\eta}) = H + K_S + 2 \eta.$$
Then 
$$c_1(\E) \cdot H = (c_1(\G_{\chi}) + 2H) \cdot H = 3H^2+H \cdot K_S$$
and 
$$c_2(\E) = c_2(\G_{\chi}) + c_1(\G_{\chi}) \cdot H + H^2 = \frac{5}{2} H^2 + \frac{3}{2} H \cdot K_S + 2 \chi(\O_S) = $$
$$\hskip .7cm = \frac{1}{2} (c_1(\E)^2 - c_1(\E) \cdot K_S) - 2(H^2-\chi(\O_S))$$
and therefore $\E$ satisfies the conditions (2.2) in \cite[Prop.~2.1]{c1}. Moreover notice that
$$H^0(\E(-H))=H^0(\G_{\chi})=0$$ 
and 
$$H^0(\E^*(2H+K_S))=H^0(\G_{\chi}^*(H+K_S))= H^0(\G_{\chi}(-c_1(\G_{\chi}))(H+K_S)) = H^0(\G_{\chi}(-2 \eta)) = 0$$ 
and we are done by \cite[Prop.~2.1]{c1}.

Furthermore, if $\chi \ge 1$, it follows by (C) that $\G_{\chi}$ is simple and then so is $\E$. If $\chi=0$ and $q(S) \ge 2$ (which implies that $p_g(S) \ge 1$) or if $\chi=0, q(S)=1, h^0(\O_S(-K_S))=0$ and $h^0(\O_S(H-K_S)) \le h^0(\O_S(H))$, then $\G_0$ is simple by Lemma \ref{e0}, whence so is $\E$.
\end{proof}
\renewcommand{\proofname}{Proof}

\renewcommand{\proofname}{Proof of Corollary \ref{main2}}
\begin{proof} 
We first prove that there exists a simple Ulrich vector bundle $\E_0$ on the minimal model $S_0$ of $S$. Let $H_0$ be a sufficiently ample divisor on $S_0$. If $q(S)=p_g(S)=0$ the existence of $\E_0$ follows by \cite[Thm.~1.2]{c1}, while if $q(S)=0$ and $p_g(S)\ge 1$ the existence of $\E_0$ follows by \cite[Thm.~1.1]{c2}. Observe that if $q(S) \ge 2$ and $\Pic(S_0) \cong \ZZ$, we have that $S_0$ is of maximal Albanese dimension, for otherwise any fiber the Albanese fibration will be a curve of self-intersection $0$, a contradiction. Therefore the existence of $\E_0$ when $q(S) \ge 1$ follows by Theorem \ref{main}. Given $\E_0$, just apply \cite[Thm.~0.1]{k} (see also \cite[Thm.~2]{s} for a more precise version) to get a simple Ulrich vector bundle $\E$ on $S$.
\end{proof}
\renewcommand{\proofname}{Proof}

\begin{remark}
\label{ipo}
Let $S$ be a smooth irreducible irregular surface and let $H$ be a very ample divisor on $S$. We have:
\begin{itemize}
\item [(i)] If $H-K_S$ is big and nef then (a) is satisfied;
\item [(ii)] The hypothesis $H^2(\O_S(H))=0$ holds, for example, whenever $S$ is not of general type or if $H^1(\O_S(H))=0$ and either $H^2 \ge H \cdot K_S$ or $\chi(\O_S) \le 5$. In particular we have that $H^1(\O_S(H))=0$ and (b) imply $H^2(\O_S(H))=0$;
\item [(iii)] If $H^1(\O_S(H))=0$ and $p_g(S) \le 5$ then (b) is satisfied;
\item [(iv)] If $S$ is not of general type or if $H^2 > 2H \cdot K_S$, then (c) is satisfied;
\item [(v)] If $H$ is sufficiently ample then (a), (b) and (c) are satisfied.
\end{itemize}
\end{remark}
\begin{proof} 
If $H-K_S$ is big and nef then $H^i(\O_S(H))=H^i(\O_S(K_S+(H-K_S)))=0$ for $i \ge 1$ by Kawamata-Viehweg vanishing and (a) is satisfied. This gives (i). If $S$ is not of general type, then $K_S$ is not big, whence $H^0(\O_S(2K_S-H))=0$ and $h^2(\O_S(H)) = h^0(\O_S(K_S-H))=0$, giving the first part of (ii) and of (iv). If $H^2 > 2H \cdot K_S$ then $H^0(\O_S(2K_S-H))=0$, and we get (iv). To see the rest of (ii) suppose that $H^2(\O_S(H)) \neq 0$, so that $H^0(\O_S(K_S-H)) \neq 0$, giving $H \cdot (K_S - H) \ge 0$. If $H^2 \ge H \cdot K_S$ we have that $H \cdot (K_S - H)=0$ and then $K_S \sim H$, giving the contradiction $0 = h^1(\O_S(H)) = h^1(\O_S(K_S)) = q(S)$. Since $H$ is very ample and $S$ is irregular we have that $h^0(\O_S(H)) \ge 5$, whence, by Riemann-Roch,
$$H^2 - H \cdot K_S \ge 2h^0(\O_S(H)) - 2 \chi(\O_S) \ge 10 - 2 \chi(\O_S).$$
If $\chi(\O_S) \le 5$ we get again $H^2(\O_S(H)) = 0$. This gives (ii). If $p_g(S) \le 5$ we get (b), whence (iii). To see (v) note that if $H = m A$ with $A$ ample and $m \gg 0$ then (a) and (b) hold trivially, and so does (c) since we can assume that $H \cdot (2K_S-H)<0$.
\end{proof}

\begin{remark}
\label{ipo4}
One may wonder if there are other surfaces, aside from surfaces of maximal Albanese dimension or $q(S)=1$, such that $H^1(\eta) = 0$ for a general $\eta \in Pic^0(S)$. The answer is no.
\end{remark}
\begin{proof} 
Suppose that such an $\eta$ exists but $S$ is not of maximal Albanese dimension. By \cite[Prop.~V.15]{b4} the Albanese map $a$ has connected fibers and has image a smooth curve $B$ of genus $q(S)$. Now $a^* : \Pic^0(B) \to \Pic^0(S)$ is injective since $a_*\O_S \cong \O_B$ and $\dim \Pic^0(B) = \dim \Pic^0(S) = q(S)$, whence $a^*$ is surjective. Therefore there exists $\eta_0 \in \Pic^0(B)$ such that $\eta = a^* \eta_0$. Since $\eta \not\cong \O_S$ (otherwise $H^1(\eta)=H^1(\O_S)=0$), we find that $\eta_0 \not\cong \O_B$ and then $h^0(\eta_0)=0$. By Riemann-Roch this gives $h^1(\eta_0)=q(S)-1$. Now the Leray spectral sequence gives the exact sequence
$$0 \to H^1(a_*\eta) \to H^1(\eta) \to H^0(R^1a_*\eta) \to 0$$
whence that $H^1(\eta_0)=H^1(a_*\eta)=0$ and therefore $q(S)=1$.
\end{proof}

\begin{remark}
\label{ipo2}
It is not difficult to see that the hypothesis $h^0(\O_S(-K_S)) = 0$ and $h^0(\O_S(H-K_S)) \le h^0(\O_S(H))$ in Theorem \ref{main} can be replaced by $h^0(\O_S(-K_S-H))=0, h^0(\O_S(-K_S)) \le h^0(\O_S(H))$ and $h^0(\O_C(H-K_S)) \le 2h^0(\O_S(H))$ for a smooth $C \in |H|$.
\end{remark}

\begin{remark}
\label{ipo3}
Using the same method in Theorem \ref{main} it follows that if $S$ is a regular surface and $H$ is a very ample non-special divisor on $S$ with $H \neq K_S, H^2 \ge H \cdot K_S$, satisfying (c) and either $h^0(\O_S(-K_S)) = 0, h^0(\O_S(H-K_S)) \le h^0(\O_S(H))$ or $h^0(\O_S(-K_S-H))=0, h^0(\O_S(-K_S)) \le h^0(\O_S(H))$ and $h^0(\O_C(H-K_S)) \le 2h^0(\O_S(H))$ for a smooth $C \in |H|$, then there exists a rank two Ulrich simple vector bundle $\E$ for the pair $(S,H)$. This does not give anything new because, since $H$ is non-special, rank two simple Ulrich vector bundles are known to exist when $p_g(S)=0$ by \cite[Thm.~1.1]{c1} and when $p_g(S) \ge 1$ by \cite[Thm.~1.1]{c2} as soon as $H^2 \ge H \cdot K_S-4$ and (c) holds.
\end{remark}

\section{Examples}

We give in this section some examples of pairs $(S,H)$ satisfying the hypotheses of Theorem \ref{main}. 

\subsection{Surfaces of maximal Albanese dimension}

Let $S$ be a surface of maximal Albanese dimension and let $H$ be a very ample divisor on $S$. 

By Remark \ref{ipo} we see that all the hypotheses of Theorem \ref{main} are satisfied if either $H$ is sufficiently ample or $H$ is non-special, $p_g(S) \le 5$ and $S$ is not of general type. Hence, under the above hypotheses on $H$, there exists a simple rank two Ulrich vector bundle $\E$ for the pair $(S,H)$. 

\subsection{Weierstrass fibrations}
\label{wei}

Let $E$ be an elliptic curve and let $\pi: S \to E$ be a Weierstrass fibration (see \cite{mi}) and let $H$ be a very ample divisor on $S$. 

Recall $\pi$ has a section $C$ with $C^2 = - n \le 0$. We will suppose that $n \ge 1$. We have $q(S)=1, \chi(\O_S)= n$ and $K_S \equiv nf$, where $f$ is a general fiber. In particular $S$ is not of general type. Assume that $H \equiv a C + bf$. It follows by \cite[Lemma 3.9(i) and (vi)]{lms} that $a \ge 3, b \ge an+1$ and $H^1(\O_S(H))=0$. Also $H^2 - H \cdot K_S = a(2b-an-n) \ge 0$ so that (b) holds and so does (a) by Remark \ref{ipo}(ii). Also (c) is satisfied by Remark \ref{ipo}(iv). Hence for all $H \equiv a C + bf$ there exists a simple rank two Ulrich vector bundle $\E$ for the pair $(S,H)$.

In the case of regular Weierstrass fibrations $\pi: S \to \PP^1$ with $n \ge 1$ it follows by \cite[Lemma 3.9(i), (v) and (vi)]{lms} that $H \equiv a C + bf$ is very ample if and only if $a \ge 3$ and $b \ge an+1$ and also that $H$ is non-special. The existence of a simple rank two Ulrich vector bundle with $H \equiv a C + bf$ was proved under some conditions on $b$ in \cite[Thm.~4]{mp} and follows in general by \cite[Thm.~1.1]{c1} for $n=1$ and by \cite[Thm.~1.1]{c2} for $n \ge 2$. 

\subsection{Surfaces geometrically ruled over an elliptic curve}

Let $S$ be a geometrically ruled surface over an elliptic curve and let $H$ be a very ample divisor on $S$. 

First $q(S)=1$ and $\chi(\O_S)=0$. We have a map $\pi: S \cong \PP(\E) \to E$ for a rank two vector bundle $\E$ on an elliptic curve $E$. We assume that $\E$ is normalized \cite[Not.~V.2.8.1]{h} with invariant $e$, so that, by \cite[Thm.~V.2.12 and Thm.~V.2.15]{h}, all cases $e \ge -1$ occur. Let $C_0$ be a section and $f$ be a fiber. By \cite[Prop.~V.2.20 and V.2.21]{h} we have that $H \equiv aC_0+bf$ with $a \ge 1$ and either $e \ge 0, b > ae$ or $e=-1, b > - \frac{a}{2}$. Now $H-K_S \equiv (a+2)C_0+(b+e)f$ and, using \cite[Prop.~V.2.20 and V.2.21]{h}, it is easy to see that either $e \le 0$ and $H-K_S$ is ample, and therefore (a) holds by Remark \ref{ipo}(i), or $e \ge 1$. In the latter case $\E$ is decomposable by \cite[Thm.~V.2.15]{h} and it is easily verified that $H^1(\O_S(H)) = 0$. Hence, by Remark \ref{ipo}(ii), we see that in all cases (a) of Theorem \ref{main} holds. Hence there exists a rank two Ulrich vector bundle $\E$ for the pair $(S,H)$.

The existence of rank two Ulrich vector bundles in this case was previously known, in the case $e>0$, by \cite[Prop.~3.1, Prop.~3.3 and Thm.~3.4]{acm} and in all cases by \cite[Cor.~3.3]{c4} (observing that, as above, one always has that $H^1(\O_S(H)) = 0$).

\subsection{Bielliptic surfaces}

Let $S$ be a bielliptic surface and let $H$ be a very ample divisor on $S$. 

Then $q(S) = 1, \chi(\O_S)=0$ and $K_S \equiv 0$. Since $H-K_S$ is ample we see by Remark \ref{ipo}(i) that (a) of Theorem \ref{main} is satisfied. As a matter of fact the Ulrich vector bundle obtained is the same as the one in \cite[Cor.~3.3]{c4} and essentially the same as the one in \cite[Prop.~6]{b1}.

\subsection{Surfaces with $p_g=0, q = 1$}

Let $S$ be a surface $p_g(S)=0, q(S) = 1$ and let $H$ be a very ample divisor on $S$. 

Since $\chi(\O_S)=0$ we have that $S$ is not of general type whence (a) of Theorem \ref{main} holds by Remark \ref{ipo} as soon as either $S$ is minimal non-ruled (because then $K_S \equiv 0$ by \cite[Prop.~VI.15 and Appendix A]{b4}) or $H$ is non-special. As a matter of fact the Ulrich vector bundle obtained is the same as the one in \cite[Cor.~3.3]{c4}.

\end{document}